\documentclass[12pt,reqno]{amsart}
\usepackage{amssymb, latexsym}
\usepackage{url}
\usepackage{a4wide}

\begin{document}
 \bibliographystyle{plain}

 \newtheorem*{KT}{Khintchine's Theorem}
 \newtheorem*{DSC}{Duffin-Schaeffer Conjecture}
 \newtheorem*{MTP}{Mass Transference Principle}
 \newtheorem{theorem}{Theorem}
 \newtheorem{lemma}{Lemma}
 \newtheorem{corollary}{Corollary}
 \newtheorem{problem}{Problem}
 \newtheorem{conjecture}{Conjecture}
 \newtheorem{definition}{Definition}
 \newcommand{\mc}{\mathcal}
 \newcommand{\rar}{\rightarrow}
 \newcommand{\Rar}{\Rightarrow}
 \newcommand{\lar}{\leftarrow}
 \newcommand{\lrar}{\leftrightarrow}
 \newcommand{\Lrar}{\Leftrightarrow}
 \newcommand{\zpz}{\mathbb{Z}/p\mathbb{Z}}
 \newcommand{\mbb}{\mathbb}
 \newcommand{\A}{\mc{A}}
 \newcommand{\B}{\mc{B}}
 \newcommand{\cc}{\mc{C}}
 \newcommand{\D}{\mc{D}}
 \newcommand{\E}{\mc{E}}
 \newcommand{\F}{\mc{F}}
 \newcommand{\G}{\mc{G}}
 \newcommand{\HH}{\mc{H}}
  \newcommand{\ZG}{\Z (G)}
 \newcommand{\FN}{\F_n}
 \newcommand{\I}{\mc{I}}
 \newcommand{\J}{\mc{J}}
 \newcommand{\M}{\mc{M}}
 \newcommand{\nn}{\mc{N}}
 \newcommand{\qq}{\mc{Q}}
 \newcommand{\U}{\mc{U}}
 \newcommand{\X}{\mc{X}}
 \newcommand{\Y}{\mc{Y}}
 \newcommand{\itQ}{\mc{Q}}
 \newcommand{\C}{\mathbb{C}}
 \newcommand{\R}{\mathbb{R}}
 \newcommand{\N}{\mathbb{N}}
 \newcommand{\Q}{\mathbb{Q}}
 \newcommand{\Z}{\mathbb{Z}}
 \newcommand{\ff}{\mathfrak F}
 \newcommand{\fb}{f_{\beta}}
 \newcommand{\fg}{f_{\gamma}}
 \newcommand{\gb}{g_{\beta}}
 \newcommand{\vphi}{\varphi}
 \newcommand{\whXq}{\widehat{X}_q(0)}
 \newcommand{\Xnn}{g_{n,N}}
 \newcommand{\lf}{\left\lfloor}
 \newcommand{\rf}{\right\rfloor}
 \newcommand{\lQx}{L_Q(x)}
 \newcommand{\lQQ}{\frac{\lQx}{Q}}
 \newcommand{\rQx}{R_Q(x)}
 \newcommand{\rQQ}{\frac{\rQx}{Q}}
 \newcommand{\elQ}{\ell_Q(\alpha )}
 \newcommand{\oa}{\overline{a}}
 \newcommand{\oI}{\overline{I}}
 \newcommand{\dx}{\text{\rm d}x}
 \newcommand{\dy}{\text{\rm d}y}
\newcommand{\cal}[1]{\mathcal{#1}}
\newcommand{\cH}{{\cal H}}
\newcommand{\diam}{\operatorname{diam}}

\parskip=0.5ex

\title[The metric theory of $p-$adic approximation]{The metric theory\\of $p-$adic approximation}
\author{Alan~K.~Haynes}
\thanks{Research supported by EPSRC grant EP/F027028/1}
\address{Department of Mathematics, University of York, Heslington, York YO10 5DD, UK}
\email{akh502@york.ac.uk}

\allowdisplaybreaks


\begin{abstract}
Metric Diophantine approximation in its classical form is the study of how well almost all real numbers can be approximated by rationals. There is a long history of results which give partial answers to this problem, but there are still questions which remain unknown. The Duffin-Schaeffer Conjecture is an attempt to answer all of these
questions in full, and it has withstood more than fifty years of mathematical investigation. In this paper we establish a strong connection between the Duffin-Schaeffer Conjecture and its $p-$adic analogue. Our main theorems are transfer principles which allow us to go back and forth between these two problems. We prove that if the variance method from probability theory can be used to solve the $p-$adic Duffin-Schaeffer Conjecture for even one prime $p$, then almost the entire classical Duffin-Schaeffer Conjecture would follow. Conversely if the variance method can be used to prove the classical conjecture then the $p-$adic conjecture is true for all primes. Furthermore we are able to unconditionally and completely establish the higher dimensional analogue of this conjecture in which we allow simultaneous approximation in any finite number and combination of real and $p-$adic fields, as long as the total number of fields involved is greater than one. Finally by using a mass transference principle for Hausdorff measures we are able to extend all of our results to their corresponding analogues with Haar measures replaced by the Hausdorff measures associated with arbitrary dimension functions.
\end{abstract}


\maketitle

\section{Introduction}
Throughout this paper we will use the following notation from elementary number theory: $\mu (n)$ is the M\"{o}bius function, $\varphi (n)$ is the Euler phi
function, $\omega (n)$ denotes the number of distinct prime divisors of $n$, $p$ and $q$ denote prime numbers,  and unless otherwise obvious $(a,n)$ denotes the greatest common divisor of $a$ and $n$. Also we use $\lambda$ to denote Lebesgue measure
on $\R/\Z$, and $\psi$ to denote a non-negative function from $\N$ to $\R$.

The Duffin-Schaeffer Conjecture is a central open problem in metric number theory which seeks to determine how well almost all real numbers can be approximated by rationals. If $\psi:\N\rar\R$ is a non-negative arithmetical function then for each $n\in\N$ we can define a subset $A_n=A_n(\psi)\subseteq\R/\Z$ by
\begin{equation*}
A_n=\bigcup_{\substack{a=1\\(a,n)=1}}^n\left[\frac{a}{n}-\psi (n),\frac{a}{n}+\psi (n)\right].
\end{equation*}
If a point $x\in\R/\Z$ falls in one of the sets $A_n$ then it means that the reduced fractions with denominators $n$ are `good' approximations to $x$, where the adjective `good' is quantified by our choice of the function $\psi$. From the point of view of Diophantine approximation it is interesting to study the collection of points in $\R/\Z$ which have infinitely many such approximations for a given $\psi$. For this reason we define a subset $W_\infty(\psi)\subseteq\R/\Z$ by
\begin{align*}
W_\infty(\psi)=\{x\in\R/\Z :x\in A_n(\psi) \text{ for infinitely many }n\in\N\}.
\end{align*}
A major stepping stone in this type of problem was provided by the following theorem, published by Khintchine in 1924 (\cite{Khintchine1924}, \cite{Khintchine1926}).
\begin{KT}
If $n^2\psi(n)$ is non-increasing then $\lambda (W_\infty(\psi))=1$ if
\begin{equation*}
\sum_{n\in\N}n\psi (n)=\infty,
\end{equation*}
and $\lambda (W_\infty(\psi))=0$ otherwise.
\end{KT}
For example Khintchine's Theorem shows that for almost every real number $x$ there are infinitely many fractions $a/n$, written in lowest terms, for which
\[\left|x-\frac{a}{n}\right|\le\frac{1}{n^2\log n\log\log n}.\]
Removing monotonicity from Khintchine's Theorem is in general very difficult. This appears to have been first addressed by R.~J.~Duffin and A.~C.~Schaeffer in a 1941 paper (\cite{DuffinSchaeffer}), and the concluding comments of that paper have inspired the following conjecture.
\begin{DSC}
For any non-negative function $\psi:\N\rar\R$ we have that $\lambda (W_\infty (\psi))=1$ if
\begin{equation}\label{DSdivcond1}
\sum_{n\in\N}\lambda (A_n)=\infty
\end{equation}
and $\lambda (W_\infty(\psi))=0$ otherwise.
\end{DSC}
We point out that when the sum on left of (\ref{DSdivcond1}) converges it follows immediately from the Borel-Cantelli Lemma that $\lambda (W_\infty(\psi))=0$. Furthermore it is a theorem of P.~Gallagher (\cite{Gallagher}) that the measure of the set $W_\infty(\psi)$ is always either one or zero. Although progress toward proving the Duffin-Schaeffer Conjecture has been made by many mathematicians, it is still the central open problem in metric number theory.

Many authors have explored various generalizations of the Duffin-Schaeffer Conjecture, but it is not within our scope to list them all here. In this paper we are going to present the natural analogues of this conjecture in the local fields $\Q_p$. Surprisingly this problem has not been directly addressed, even though its counterpart for non-Archimedean local fields of positive characteristic (i.e. fields of formal Laurent series over finite fields) has already been studied (\cite{InoueNakada}, \cite{Inoue}). However, from a broader point of view, the metric theory of Diophantine approximation in $\Q_p$ itself is by no means an unexplored topic. It appears to have been first considered by K.~Mahler in 1940 (\cite{Mahler}). Shortly thereafter a $p-$adic analogue of Khintchine's Theorem was proved by V.~Jarnik (\cite{Jarnik1945}) and then generalized to higher dimensional linear forms in a 1955 thesis by E.~Lutz (\cite{Lutz}). Stemming from Lutz's work there have been several other advances in understanding this subject (see \cite{BernikDodson}, \cite{Bugeaud}, and \cite{BDV06}). However it does not appear that anyone has been successful in allowing non-monotonic approximating functions in the $p-$adic setting. A naive attempt to do this with Lutz's setup leads to serious issues, which we will address in detail in Section \ref{zeroone}. For now it will suffice to say that these issues demand certain modifications in the setup of the problem. Here we present what we believe is the most natural $p-$adic version of the Duffin-Schaeffer Conjecture.

Recall that for any prime $p$ we can define an absolute value $| ~|_p$ on $\Q$ by writing $a/n\in\Q$ as
\[\frac{a}{n}=\frac{p^s b}{m}~\text{ with }~s\in\Z, ~p\not| ~b,m,\]
and then setting \[\left|\frac{a}{n}\right|_p=p^{-s}.\]
The field $\Q_p$ is then defined as the completion of $\Q$ with respect to $|~|_p$, and the ring of integers $\Z_p\subseteq\Q_p$ is defined by
\[\Z_p=\{x\in\Q_p : |x|_p\le 1\}.\]
The ring $\Z_p$ is a compact group under addition, so there is a unique Haar probability measure on $\Z_p$ which we will denote by $\mu_p$.

Now if $p$ is a prime and if $\psi:\N\rar\R$ is non-negative then for each $n\in\N$ we define a set $\E_n=\E_n(\psi)\subseteq\Z_p$ by
\begin{equation*}
\E_n=\bigcup_{\substack{a=-n\\(a,n)=1}}^n\left\{x\in\Z_p: \left|x-\frac{a}{n}\right|_p\le\psi (n)\right\}.
\end{equation*}
It is then reasonable to consider the set $W_p(\psi)\subseteq\Z_p$ defined by
\[W_p(\psi)=\{x\in\Z_p :x\in\E_n(\psi) \text{ for infinitely many } n\in\N\}.\]

There are initially two justifications for choosing the sets $\{\E_n\}$ as our sets of `good' approximations in $\Z_p$:
\begin{itemize}
\item[(i)] The collection of points
\begin{equation}\label{F_neqn}
\{a/n\in\Q : |a|\le n, (a,n)=1\}
\end{equation}
is dense in $\Z_p$. This will follow from Lemma \ref{largepsilemma} below.
\item[(ii)] As in the classical problems the quality of approximation required of an element of $\Q$ depends on it's complexity, which we measure as the element's absolute Weil height (i.e. the maximum of the absolute values of it's numerator and denominator in reduced form). The reason for disallowing non-reduced fractions is mentioned in the next section.
\end{itemize}
It is logical to argue that for $n\in\N$ the collection of fractions
\[\{a/n\in\Q : |a|\le n, (a,n)=1\}\cup\{n/a\in\Q : |a|\le n, (a,n)=1\}\]
is the complete choice of approximations of height $n$ in $\Q_p$. Although this choice would bring us closer to Jarnik and Lutz's original setup, it would also require us to abandon the possibility of a zero-one law (this is also demonstrated in the next section).

The $p-$adic version of the Duffin-Schaeffer Conjecture which we wish to consider is the following problem.
\begin{conjecture}\label{Q_pDSC}
For any prime $p$ and for any non-negative function $\psi:\N\rar\R$ we have that $\mu_p (W_p (\psi))=1$ if
\begin{equation}\label{Q_pDSdivcond1}
\sum_{n\in\N}\mu_p (\E_n)=\infty
\end{equation}
and $\mu_p (W_p(\psi))=0$ otherwise.
\end{conjecture}
We will establish a strong connection between this conjecture and the Duffin-Schaeffer Conjecture. In order to present our main results we must introduce the concept of quasi-independence on average. For our considerations it will be sufficient to make the following two definitions.
\begin{enumerate}
\item[(i)] Suppose $\psi:\N\rar\R$ is a non-negative function. We will say that {\em quasi-independence on average holds in $\R$ for $\psi$}, and abbreviate this by saying that {\em ($QIA_\infty ,\psi$) holds}, if
    \begin{equation*}
    \limsup_{N\rar\infty}\left(\sum_{n\le N}
    \lambda(A_n)\right)^2\left(\sum_{m,n\le N}
    \lambda(A_m\cap A_n)\right)^{-1}>0.
    \end{equation*}
    If ($QIA_\infty ,\psi$) holds for all $\psi$ which satisfy (\ref{DSdivcond1})
    then we will say that {\em ($QIA_\infty$) holds}.
\item[(ii)] If $p$ is a prime then we will say that {\em quasi-independence on average holds in $\Q_p$ for $\psi$}, and abbreviate this by saying that {\em ($QIA_p ,\psi$) holds}, if
    \begin{equation*}
    \limsup_{N\rar\infty}\left(\sum_{n\le N}
    \mu_p(\E_n)\right)^2\left(\sum_{m,n\le N}
    \mu_p(\E_m\cap \E_n)\right)^{-1}>0.
    \end{equation*}
    If ($QIA_p ,\psi$) holds for all $\psi$ which satisfy
    (\ref{Q_pDSdivcond1}) then we will say that {\em ($QIA_p$) holds}.
\end{enumerate}

It is well known that if $\psi:\N\rar\R$ is a non-negative function which satisfies (\ref{DSdivcond1}) then in order to show that $\lambda (W_\infty (\psi))=1$ it is sufficient to show that ($QIA_\infty, \psi$) holds (see \cite[Lemma 2.3]{HarmanMNT}). {\em All currently known results about the Duffin-Schaeffer Conjecture can be proved by using this fact.} Similarly if $p$ is a prime and if $\psi$ satisfies (\ref{Q_pDSdivcond1}) then if ($QIA_p, \psi$) holds it follows readily that $\mu_p(W_p(\psi))>0$. We can get from there to full measure with the aid of our zero-one law, Lemma \ref{zeroonelem} below. In light of these observations we now present the first main results of this paper.
\begin{theorem}\label{Q_pRtransferthm}
For any fixed prime $p$ if ($QIA_p$) holds then the Duffin-Schaeffer Conjecture is true in all cases when the function $\psi$ is supported on a set $S\subseteq\N$ with the property that there exists an $N\in\N$ for which $p^{N+1}\not| ~ n$ for any $n\in S$.
\end{theorem}
\begin{corollary}\label{Q_pRtransfercor}
If ($QIA_p$) holds for some prime $p$ then the Duffin-Schaeffer Conjecture is true in all cases when the function $\psi$ is supported on the set of squarefree integers.
\end{corollary}
\begin{theorem}\label{RQ_ptransferthm}
If ($QIA_\infty$) holds then Conjecture \ref{Q_pDSC} is true for every prime $p$.
\end{theorem}
After proving Theorems \ref{Q_pRtransferthm} and \ref{RQ_ptransferthm} it will be relatively easy to deduce nontrivial results about Conjecture \ref{Q_pDSC} itself. As in the classical case one direction of the conjecture follows trivially from the Borel-Cantelli Lemma, and the nontrivial case is when (\ref{Q_pDSdivcond1}) holds. By transferring known results about the Duffin-Schaeffer Conjecture to the fields $\Q_p$ we will prove the following theorem.
\begin{theorem}\label{dimonethm}
Suppose that $p$ is a prime and that $\psi$ satisfies (\ref{Q_pDSdivcond1}). Then we have that $\mu_p(W_p(\psi))=1$ whenever at least one of the following conditions is satisfied.
\begin{enumerate}
\item[(i)] We have that \[\sum_{\substack{n\in\N\\ \psi (n)\ge 1/n}}\mu_p(\E_n)=\infty.\]
\item[(ii)] We have that
\begin{equation*}
\limsup_{N\rar\infty}\frac{\sum_{n\le N}\varphi (n)\psi (n)}{\sum_{n\le N}n\psi (n)}>0.
\end{equation*}
\item[(iii)] For some $\gamma\in\R$ the sequence ${n^{-\gamma}\psi (n)}$ is non-increasing on the set $\N\setminus p\N$.
\item[(iv)] We have that $\psi (n)\ll n^{-2}$ as $n\rar\infty$.
\item[(v)] For some $\epsilon >0$,
\begin{equation*}
\sum_{n\in\N}\varphi (n)(\psi (n))^{1+\epsilon}=\infty.
\end{equation*}
\end{enumerate}
\end{theorem}
Furthermore we will also prove a higher dimensional result in which we consider simultaneous approximation in combinations of fields from the collection $\R\cup\{\Q_p:p \text{ prime}\}$. Suppose that $\ell$ is a non-negative integer, $k\in\N$, and  $p_1,\ldots ,p_k$ are (not necessarily distinct) primes. We can define a norm $|~|_{\ell, p_1,\ldots,p_k}$ on the vector space $\R^{\ell}\times\Q_{p_1}\times\cdots\times\Q_{p_k}$ over $\Q$ by
\begin{equation*}
|(x_1,\ldots ,x_{k+\ell})|_{\ell, p_1,\ldots,p_k}=\max\{\|x_1\|,\ldots ,\|x_\ell\|,|x_{\ell+1}|_{p_1},\ldots ,|x_{\ell+k}|_{p_k}\},
\end{equation*}
where for $x\in\R$ we have set $\|x\|=\min_{a\in\Z}|x-a|$. Since the absolute value $\|~\|$ is well defined modulo $\Z$ we may also regard $|~|_{\ell, p_1,\ldots,p_k}$ as a norm on $(\R/\Z)^{\ell}\times\Z_{p_1}\times\cdots\times\Z_{p_k}$. A natural probability measure on the latter space is the product measure
\[\mu_{\ell, p_1,\ldots,p_k}=\lambda^\ell\times\prod_{i=1}^k\mu_{p_i}.\]

Now for a non-negative arithmetical function $\psi$ and for each positive integer $n$ we may define a set $\E_{n}^{\ell,p_1,\ldots ,p_k}=\E_{n}^{\ell,p_1,\ldots ,p_k}(\psi)\subseteq (\R/\Z)^{\ell}\times\Z_{p_1}\times\cdots\times\Z_{p_k}$ by
\begin{align*}
\E_{n}^{\ell,p_1,\ldots ,p_k}=\bigcup_{\substack{a_1=-n\\ (a_1,n)=1}}^n\cdots\bigcup_{\substack{a_{\ell+k}=-n\\ (a_{\ell+k},n)=1}}^n&\Big\{(x_1,\ldots ,x_{\ell+k})\in (\R/\Z)^{\ell}\times\Z_{p_1}\times\cdots\times\Z_{p_k}:\\
&\qquad\left|\left(x_1-\frac{a_1}{n},\ldots ,x_{\ell+k}-\frac{a_{\ell+k}}{n}\right)\right|_{\ell, p_1,\ldots,p_k}\le\psi (n)\Big\},
\end{align*}
and we may then define $W_{\ell, p_1,\ldots,p_k}(\psi)\subseteq (\R/\Z)^{\ell}\times\Z_{p_1}\times\cdots\times\Z_{p_k}$ by
\begin{align*}
W_{\ell, p_1,\ldots,p_k}(\psi)=\{(x_1,\ldots ,x_{\ell+k}) :x\in \E_{n}^{\ell,p_1,\ldots ,p_k}(\psi) \text{ for infinitely many }n\in\N\}.
\end{align*}
In Section \ref{Q_presultssec} we will prove the following theorem.
\begin{theorem}\label{highdimthm}
Suppose that $\ell$ is a non-negative integer, that $k\in\N$, and that $\ell+k>1$. If $p_1,\ldots ,p_k$ are (not necessarily distinct) primes then we have that
\begin{equation*}
\mu_{\ell, p_1,\ldots,p_k}(W_{\ell, p_1,\ldots,p_k}(\psi))=1
\end{equation*}
if and only if
\begin{equation*}
\sum_{n=1}^\infty \mu_{\ell, p_1,\ldots,p_k}(\E_n^{\ell, p_1,\ldots,p_k})=\infty.
\end{equation*}
\end{theorem}
This theorem is a generalization of the well known Pollington-Vaughan Theorem proved in \cite{VaughanPollington}.

In Section \ref{hausdorffsection} we explore the natural Hausdorff measure generalizations of our problems. By using recent results due to V.~Beresnevich and S.~Velani (\cite{BeresnevichVelani06}) we show that if Conjecture \ref{Q_pDSC} is true then all of the corresponding Hausdorff measure conjectures for $\Q_p$ follow (Theorem \ref{Q_pMTPthm2}). Similarly we show how Theorem \ref{highdimthm} implies all of its Hausdorff measure analogues (Theorem \ref{Q_pMTPthm1}). These results may come as somewhat of a surprise because the Hausdorff measure variants of conjectures of this type are usually thought of as refinements of the original problems.

Finally in Section \ref{conclusionsec} we present three problems which we hope will inspire further research into the real and $p-$adic versions of the Duffin-Schaeffer Conjecture.

{\em Acknowledgements:} For their encouragement and support I would like to thank Sanju Velani and Victor Beresnevich.

\section{Zero-one laws}\label{zeroone}
As a point of reference we begin from Jarnik and Lutz's original setup, although with a different notation. If $p$ is a prime and $\psi$ is a non-negative arithmetical function then for each $n\in\N$ we may define $\E_n'=\E_n'(\psi)\subseteq\Z_p$ by
\begin{equation*}
\E_n'=\bigcup_{a=-n}^n\left(\left\{x\in\Z_p: \left|x-\frac{a}{n}\right|_p\le\psi (n)\right\}\cup\left\{x\in\Z_p: \left|x-\frac{n}{a}\right|_p\le\psi (n)\right\}\right),
\end{equation*}
and we may then define $W_p'=W_p'(\psi)\subseteq\Z_p$ by
\[W_p'(\psi)=\{x\in\Z_p :x\in\E_n'(\psi) \text{ for infinitely many } n\in\N\}.\]
In \cite{Jarnik1945} and \cite[Theorems 4.22, 4.23]{Lutz} it is proved that if $\psi$ is a monotonic function which satisfies $\psi (n)<1/2n$ and other mild regularity conditions then $\mu_p (W_p'(\psi))=1$ if
\begin{equation*}
\sum_{n\in\N}n\psi (n)=\infty,
\end{equation*}
and $\mu_p (W_p'(\psi))=0$ otherwise. This is a $p-$adic version of Khintchine's Theorem and a separate proof using ubiquitous systems can be found in \cite[Section 12.6]{BDV06}.

If we want to work with non-monotonic functions $\psi$ then there are a few things to be careful of. First of all if we do not enforce coprimeness then it is possible to choose $\psi$ so that
\begin{equation*}
\sum_{n\in\N}\mu_p (\E_n')=\infty
\end{equation*}
but $\mu_p(W_p'(\psi))=0$. The example at the end of Duffin and Schaeffer's paper \cite{DuffinSchaeffer} still applies in the $p-$adic setting to illustrate how this can happen. Taking this into account we might hope to work with the subsets $\E_n''(\psi)$ and $W_p''(\psi)$ of $\Z_p$ defined by
\begin{align*}
\E_n''(\psi)&=\bigcup_{\substack{a=-n\\(a,n)=1}}^n\left(\left\{x\in\Z_p: \left|x-\frac{a}{n}\right|_p\le\psi (n)\right\}\cup\left\{x\in\Z_p: \left|x-\frac{n}{a}\right|_p\le\psi (n)\right\}\right),\\
\intertext{and}
W_p''(\psi)&=\{x\in\Z_p :x\in\E_n''(\psi) \text{ for infinitely many } n\in\N\}.
\end{align*}
However even with these sets there is still an exotic possibility, which is the failure of the zero-one law altogether. To see this let
\[\psi(n)=\begin{cases}p^{-1}&\text{ if }p|n,\\0&\text{ otherwise.}\end{cases}\]
Then if $x\in\E_n''(\psi)$ we must have that $p|n$ and that
\[\left|x-\frac{a}{n}\right|_p\le p^{-1}~\text{ or }~\left|x-\frac{n}{a}\right|_p\le p^{-1}\]
for some $-n\le a\le n$ with $(a,n)=1$. The inequality on the left is impossible since $a/n\notin\Z_p$. This leaves us with the inequality on the right, which is equivalent to
\begin{equation*}
\left|ax-n\right|_p\le p^{-1}.
\end{equation*}
Now supposing that this is satisfied then using the strong triangle inequality and the fact that $p\nmid a$ we deduce that
\[|x|_p=|ax|_p\le\max\{|ax-n|_p,|n|_p\}\le p^{-1}.\]

Conversely if we start with any $n\in p\N$ and any $x\in\Z_p$ which satisfies $|x|_p\le p^{-1}$ then for any $-n\le a\le n$ with $(a,n)=1$ we have that
\[|ax-n|_p\le \max\{|ax|_p,|n|_p\}\le p^{-1}.\] This proves that \[\E_n''(\psi)=\{x\in \Z_p:|x|_p\le p^{-1}\}\] and so we have that $\mu_p(W_p''(\psi))=p^{-1}.$

By contrast we will show that no matter how $\psi$ is chosen the sets $W_p(\psi)$ always have measure $0$ or $1$. This is the one-dimensional case of the following lemma.
\begin{lemma}\label{zeroonelem}
Suppose that $\ell$ is a non-negative integer, that $k\in\N$, and that $p_1,\ldots ,p_k$ are (not necessarily distinct) primes. Then for any $\psi$ we have that
\begin{equation*}
\mu_{\ell, p_1,\ldots,p_k}(W_{\ell, p_1,\ldots,p_k}(\psi))=0 \text{ or } 1.
\end{equation*}
\end{lemma}
Before we give the proof of Lemma \ref{zeroonelem} we need to prove the following auxiliary result, which also serves the purpose of verifying that the collection of points (\ref{F_neqn}) is dense in $\Z_p$.
\begin{lemma}\label{largepsilemma}
If $p|n$ then $\E_n(\psi)=\emptyset$ or $\Z_p$. Also if $p\nmid n$ but $\psi (n)>4^{\omega (n)}/n$ then $\E_n (\psi)=\Z_p$.
\end{lemma}
\begin{proof}
For the proof of the first part of the lemma suppose that $p|n$ and that $x\in\Z_p$ satisfies
\[\left|x-\frac{a}{n}\right|_p\le \psi (n),\]
with $(a,n)=1$. Then since $|a/n|_p>1$ we must have that $\psi (n)>1$ and it follows that for any $y\in\Z_p$
\[\left|y-\frac{a}{n}\right|_p\le\max\left\{|y-x|_p,\left|x-\frac{a}{n}\right|_p\right\}\le\psi (n).\]

For the second part of the lemma we are given that $p\nmid n$. Let $x\in\Z_p$ and suppose that the $p-$adic expansion of $nx$ is
\[nx=\sum_{m=0}^\infty b_mp^m~\text{ with each }~b_m\in\{0,\ldots , p-1\}.\]
If $\psi (n)> 1$ then the proof is trivial, so without loss of generality assume that $\psi (n)\in (p^{-(N+1)},p^{-N}]$ for some non-negative integer $N$. We have that $x\in\E_n$ if and only if there exists an integer $-n\le a\le n$ with $(a,n)=1$ which satisfies
\begin{equation*}
|nx-a|_p\le p^{-N}.
\end{equation*}
If such an integer exists then it must have a $p-$adic expansion of the form
\[a=\sum_{m=0}^{N-1}b_mp^m+\sum_{m=N}^\infty c_mp^m~\text{ with each }~c_m\in\{0,\ldots , p-1\}.\]
Now it is easy to see that this problem is equivalent to determining whether or not there is an integer $-n\le a\le n$ which satisfies \[(a,n)=1~\text{ and }~a\equiv\sum_{m=0}^{N-1}b_mp^m\mod p^N.\]
Writing $b=\sum_{m=0}^{N-1}b_mp^m$ we find that
\begin{align*}
\#\{a\in\N : |a|\le n , (a,n)=1,~ a\equiv b\mod p^N\}\\
&=\sum_{\substack{|a|\le n\\a\equiv b\mod p^N}}\sum_{d|a,n}\mu (d)\\
&=\sum_{d|n}\mu (d)\sum_{\substack{|\ell |\le n/d\\ \ell\equiv bd^{-1}\mod p^N}}1\\
&=\sum_{d|n}\mu (d)\left(\frac{2n}{dp^N}+k_d\right)\\
&=\frac{2\varphi (n)}{p^N}+\sum_{d|n}k_d\mu (d)\\
&\ge 2\varphi(n)\psi (n)+\sum_{d|n}k_d\mu (d).
\end{align*}
Here each of the quantities $k_d$ is a real number which satisfies $|k_d|\le 2$. For a bound on the right hand sum we have that
\begin{equation*}
\left|\sum_{d|n}k_d\mu (d)\right|\le 2\sum_{d|n}|\mu (d)|= 2^{\omega (n)+1}.
\end{equation*}
Now we use our hypothesis and find that
\begin{align*}
2\varphi(n)\psi (n)>\frac{2\cdot 4^{\omega (n)}\varphi (n)}{n}=2\cdot\prod_{\substack{q|n\\q\text{ prime}}}4\left(1-q^{-1}\right)\ge 2^{\omega (n)+1}.
\end{align*}
This proves that
\begin{align*}
\#\{a\in\N : |a|\le n, (a,n)=1,~ a\equiv b\mod p^N\}\ge 1,
\end{align*}
which finishes the proof of Lemma \ref{largepsilemma}. The fact that the collection of points (\ref{F_neqn}) is dense in $\Z_p$ follows from the second statement of the lemma together with the observation that $4^{\omega(n)}/n\rar 0$ as $n\rar\infty$ (see \cite[Section 22.10]{HardyWright}).
\end{proof}

Now we proceed to the proof of our zero-one law.
\begin{proof}[Proof of Lemma \ref{zeroonelem}]
First of all by Lemma \ref{largepsilemma} if $\mu (\E_n(\psi))>0$ for infinitely many $n\in p\N$ then we have that $W_p(\psi)=\Z_p$. On the other hand if $\mu (\E_n(\psi))>0$ for only finitely many $n\in p\N$ then it follows that
\[\mu_p\left(\{x\in\Z_p :x\in\E_n(\psi) \text{ for infinitely many } n\in p\N\}\right)=0.\]
Thus without loss of generality we may assume that $\psi (n)=0$ whenever $p|n$.

Now let us give the proof of Lemma \ref{zeroonelem} when $\ell=0$ and $k=1$, and let us write $p_1=p$. By a lemma due originally to Cassels (see \cite{Cassels}, \cite{Gallagher}, and \cite{HarmanMNT}) expanding or contracting each of the component intervals of the sets $\E_n(\psi)$ by a factor of $p$ will not change the measure of the resulting limsup set. In other words for any integer $i$ we have that
\[\mu_p(W_p(\psi))=\mu_p(W_p(p^i\psi)).\]
If there is an integer $i$ for which $p^i\psi (n)>1/p$ for infinitely many $n$ then by the second part of Lemma \ref{largepsilemma} we could conclude that $W_p(p^i\psi)=\Z_p$ and so $\mu_p(W_p(\psi))=1$. Thus again without loss of generality we may assume that for every integer $i$ we have $p^i\psi (n)<1/p$ for all but at most finitely many $n$. Also notice that \[W_p(\psi)\subseteq W_p(p\psi)\subseteq W_p(p^2\psi)\subseteq\cdots,\] which with the above comments implies that
\[\mu_p(W_p(\psi))=\mu_p\left(\bigcup_{i=0}^{\infty}W_p(p^i\psi)\right).\]

Next define a map $\tau_p:\Z_p\rar\Z_p$ by writing $x\in\Z_p$ as
\[x=\sum_{m=0}^\infty b_mp^m~\text{ with each }~b_m\in\{0,\ldots , p-1\}\]
and then setting
\[\tau_p (x)=\begin{cases}\sum_{m=0}^\infty b_{m+1}p^m&\text{ if }b_0=0,\text{ and }\\ 1+\sum_{m=0}^\infty b_{m+1}p^m&\text{ otherwise }.\end{cases}\]
It is easy to see that if $M$ is a positive integer and $B\subseteq\Z_p$ is a ball of radius $1/p^M$ then $\mu_p (\tau_p(B))=1/p^{M-1}.$ It follows from this that if $B'$ is a measurable subset of $\Z_p$ which is contained in a ball of radius $1/p$ then $\mu_p (\tau_p(B'))=p\cdot\mu_p(B')$.

If $x\in\E_n(\psi)$ then since $p\nmid n$ we have that
\[\left|nx-a\right|_p\le\psi (n)\]
for some $-n\le a\le n$ with $(a,n)=1$. If $\psi (n)<1/p$ then this implies that
\[nb_0\equiv a\mod p.\]
In the case when $b_0=0$ we have
\[\left|n\tau_p(x)-\frac{a}{p}\right|_p\le p\psi (n),\]
whereas if $b_0\not= 0$ then
\[\left|n\tau_p(x)-\frac{a+(p-b_0)n}{p}\right|_p\le p\psi (n).\]
In the first case $a/p$ is an integer in $[-n,n]$ which is coprime to $n$, and in the second $(a+(p-b_0)n)/p$ is such an integer. Iterating this argument we find that if $j$ is a non-negative integer then
\[\tau_p^j(W_p(\psi))\subseteq \bigcup_{i=j}^\infty W_p(p^i\psi).\]

Now suppose that $\mu_p(W_p(\psi))>0$. By a density argument which follows easily from the Lebesgue Density Theorem on $\R$, for any $\epsilon >0$ we can find integers $x_0\in\Z$ and $M\in\N$ with the properties that
\[\mu_p\left(\{x\in W_p(\psi):|x-x_0|_p\le p^{-M}\}\right)\ge (1-\epsilon)p^{-M}.\]
Then we have that
\[\tau^M\left(\{x\in W_p(\psi):|x-x_0|_p\le p^{-M}\}\right)\subseteq \left(\bigcup_{i=M}^{\infty}W_p(p^i\psi)\right)\]
and that
\[\mu_p\left(\tau^M\left(\{x\in W_p(\psi):|x-x_0|_p\le p^{-M}\}\right)\right)=p^M\cdot\mu_p\left(\{x\in W_p(\psi):|x-x_0|_p\le p^{-M}\}\right).\]
This implies that
\[\mu_p(W_p(\psi))>1-\epsilon\]
and since $\epsilon$ is arbitrary we conclude that $\mu_p(W_p(\psi))=1$.

The higher dimensional sets can be treated in much the same way. We will outline the proof and leave the details to the reader. Suppose that $\ell$ is a non-negative integer and that $k\in\N$. By induction on $k$ we may assume as before that $\psi(n)=0$ whenever $p_i|n$ for some $1\le i\le k$. We define maps $\tau_{p_i}:\Z_{p_i}\rar\Z_{p_i}$ exactly as above and we supplement this by defining $\tau_\infty:\R/\Z\rar\R/\Z$ by
\[\tau_\infty (x)=p_1x\mod 1.\]
Like the maps $\tau_{p_i}$ the map $\tau_\infty$ is metrically transitive (\cite{Gallagher}), which means that any set which is mapped into itself by $\tau_\infty$ must have measure zero or one. If $x\in A_n (\psi)$ then we have that
\[\left\|x-\frac{a}{n}\right\|\le \psi(n)\]
for some $1\le a\le n$ with $(a,n)=1$, which implies that
\[\left\|\tau_\infty(x)-\frac{p_1a}{n}\right\|\le p_1\psi(n).\]
Since we are working in $\R/\Z$ we can replace $p_1a$ by its least positive representative modulo $n$. In this way we find that $\tau_\infty(W_\infty(\psi))\subseteq W_\infty(p_1\psi)$.

Finally we define a map $\tau_{\ell,p_1,\ldots ,p_k}$ from $(\R/\Z)^{\ell}\times\Z_{p_1}\times\cdots\times\Z_{p_k}$ onto itself by setting
\[\tau_{\ell,p_1,\ldots ,p_k}(x_1,\ldots , x_\ell,x_{\ell+1},\ldots ,x_{\ell+k})=(\tau_\infty(x_1),\ldots , \tau_\infty(x_\ell),\tau_{p_1}(x_{\ell+1}),\ldots ,\tau_{p_k}(x_{\ell+k})).\] The rest of the proof can then be finished by the same arguments which we used in the one-dimensional case.
\end{proof}

\section{Overlap estimates}\label{overlapsec}
In working with conditions like ($QIA_\infty , \psi$) and ($QIA_p , \psi$) it is useful to have good bounds for the measures of the sets $A_m\cap A_n$ and $\E_m\cap\E_n$. Upper bounds of this type have been well studied (at least in $\R$) and they are referred to as {\em overlap estimates}. We will also be interested in lower estimates and it seems that a natural way to capture both bounds at the same time is by using results about arithmetic in the group ring of $\Q/\Z$. This appears in essence to be the approach used to obtain the overlap estimates in \cite{VaughanPollington}.

For each non-negative integer $n$ we can define a formal polynomial $F_n(z)$ by
\[F_n(z)=\sum_{\substack{a=1 \\ (a,n)=1}}^nz^{a/n\mod 1}.\]
The following result is proved in \cite{HaynesHomma}.
\begin{theorem}\cite[Theorem 2]{HaynesHomma}\label{ZGaddthm}
Suppose that $m$ and $n$ are positive integers. Let $d=(m,n)$ and
let $d'$ be the largest divisor of $d$ which is relatively prime
to both $m/d$ and $n/d$. Then we have that
\begin{align*}
F_m\times F_n=\varphi (d)\sum_{e|d'}c(d',e)F_{mn/de},
\end{align*}
where
\begin{equation*}
c(d',e)=\prod_{\substack{q|d',~ q\nmid e\\q\text{ prime}}}\left(1-\frac{1}{q-1}\right).
\end{equation*}
\end{theorem}
With the aid of this theorem we are able to relate the sizes of intersections of our sets in $\R/\Z$ to those in $\Z_p$ in the following way.
\begin{lemma}\label{Q_pRoverlaplem}
Suppose that $\psi$ takes values in the set $\{0,1,p^{-1},p^{-2},\ldots\}$ and further suppose that $\psi (n)<1/4n$ for all $n\in\N$. Then for all $m,n\in\N$ with $p\nmid m,n$ we have that
\begin{equation*}
\lambda (A_m(\psi/2)\cap A_n(\psi/2))\le \mu_p (\E_m(\psi)\cap\E_n(\psi))\le \frac{3}{2}\cdot\lambda (A_m(2\psi)\cap A_n(2\psi)).
\end{equation*}
\end{lemma}
\begin{proof}
The condition that $\psi (n)<1/4n$ guarantees that the intervals used in the definitions of $A_n(\psi/2)$ and $A_n(2\psi)$ are disjoint. Similarly if $p\nmid n$ then suppose that $a$ and $b$ are two positive integers which are both relatively prime to $n$ and for which
\[\left|\frac{a}{n}-\frac{b}{n}\right|_p\le\psi (n).\]
Writing $\psi (n)=p^{-N}$ we see that the above equation is satisfied if and only if
$p^N| (a-b)$. Since \[p^N=\psi (n)^{-1}>4n\] it follows that at most one of the integers $a$ and $b$ can lie in $[-n,n]$. Finally if two balls in $\Q_p$ intersect each other then one of them must be contained in the other (a fact which we will use from here on without reference), so this argument proves that the balls used in the definition of $\E_n(\psi)$ are disjoint.

If $m=n$ then the statement of the lemma is easily verified. Therefore let us assume that $m\not= n$, $\psi (m)\not=0,$ and $\psi (n)\not=0$, and set
\begin{align*}
\delta=\delta(m,n)=\min\left\{\psi (m),\psi (n)\right\}~\text{ and }~
\Delta=\Delta(m,n)=\max\left\{\psi (m),\psi (n)\right\}.
\end{align*}
Then we find immediately that
\begin{align*}
\lambda (A_m(2\psi)\cap A_n(2\psi))&\ge 2\delta\cdot\#\Big\{a,b\in\N : a\le m,~ b\le n, \\ &\qquad\qquad\qquad\qquad\qquad (a,m)=(b,n)=1,~ \left\|\frac{a}{m}-\frac{b}{n}\right\|\le2\Delta\Big\},\nonumber\\
\intertext{and that}\lambda (A_m(\psi/2)\cap A_n(\psi/2))&\le \delta\cdot\#\Big\{a,b\in\N : a\le m, ~b\le n,\\
&\qquad\qquad\qquad\qquad\qquad (a,m)=(b,n)=1,~ \left\|\frac{a}{m}-\frac{b}{n}\right\|\le\Delta\Big\}.\nonumber
\end{align*}
The quantity
\[\#\left\{a,b\in\N : a\le m,~ b\le n, ~(a,m)=(b,n)=1,~ \left\|\frac{a}{m}-\frac{b}{n}\right\|\le\Delta\right\}\]
is equal to the number of monomials $z^{\gamma\mod 1}$ which appear in the product $F_m\times F_n$ and which also satisfy $\|\gamma\|\le\Delta$. Of course these are counted with multiplicity. This observation together with Theorem \ref{ZGaddthm} gives us the bounds
\begin{align}
\lambda (A_m(2\psi)\cap A_n(2\psi))&\ge 4\delta\varphi (d)\sum_{e|d'}c(d',e)\#\left\{a\in\N :a\le\frac{2mn\Delta}{de}, \left(a,\frac{mn}{de}\right)=1\right\}\label{Roverlap1}\\
\intertext{and}
\lambda (A_m(\psi/2)\cap A_n(\psi/2))&\le 2\delta\varphi (d)\sum_{e|d'}c(d',e)\#\left\{a\in\N :a\le\frac{mn\Delta}{de}, \left(a,\frac{mn}{de}\right)=1\right\}.\label{Roverlap2}
\end{align}
For the analysis of the $p-$adic case we start with the equality
\begin{align*}
\mu_p(\E_m(\psi)&\cap \E_n(\psi))\nonumber\\
&= \delta\cdot\#\left\{a,b\in\Z : |a|\le m, |b|\le n,(a,m)=(b,n)=1, \left|\frac{a}{m}-\frac{b}{n}\right|_p\le\Delta\right\}.
\end{align*}
Our aim is to apply Theorem \ref{ZGaddthm} to estimate the right hand side from above and below. However the problem is that the exponents of the monomials in Theorem \ref{ZGaddthm} are only determined modulo one, but the absolute value $|~|_p$ is not invariant under integer translation. This means that we have to be careful to make the correct choice of representatives. We do this by setting
\begin{align*}
N_1&=\#\left\{a,b\in\Z : 1\le a\le m, 1\le b\le n,(a,m)=(b,n)=1, \left|\frac{a}{m}-\frac{b}{n}\right|_p\le\Delta\right\},\\
N_2&=\#\left\{a,b\in\Z : -m\le a\le -1, -n\le b\le -1,(a,m)=(b,n)=1, \left|\frac{a}{m}-\frac{b}{n}\right|_p\le\Delta\right\},\\
N_3&=\#\left\{a,b\in\Z : 1\le a\le m, -n\le b\le -1,(a,m)=(b,n)=1, \left|\frac{a}{m}-\frac{b}{n}\right|_p\le\Delta\right\},\\
\intertext{ and}
N_4&=\#\left\{a,b\in\Z : -m\le a\le -1, 1\le b\le n,(a,m)=(b,n)=1, \left|\frac{a}{m}-\frac{b}{n}\right|_p\le\Delta\right\}.
\end{align*}
Since $|x|_p=|-x|_p$ for all $x\in\Q_p$ it is immediate that $N_1=N_2$ and $N_3=N_4$ and thus
\[\mu_p(\E_m(\psi)\cap \E_n(\psi))=2\delta(N_1+N_3).\]
We also have that
\begin{align}
N_3&=\#\left\{a,b\in\Z : 1\le a\le m, 1\le b\le n,(a,m)=(b,n)=1, \left|\frac{a}{m}+\frac{b}{n}\right|_p\le\Delta\right\}\nonumber\\
&=\#\left\{a,b\in\Z : 1\le a\le m, 1\le b\le n,(a,m)=(b,n)=1, \left|\frac{a}{m}-\frac{b}{n}+1\right|_p\le\Delta\right\},\label{Q_poverlap6}
\end{align}
where for the second equality we have used the bijection $b/n\leftrightarrow (n-b)/n.$ Now the quantities $a/m-b/n$ which are being counted in $N_1$ all lie in the interval $(-1,1)$ on the real line. By replacing $a/m$ and $b/n$ by $(m-a)/m$ and $(n-b)/n$ we see that there is a symmetry about zero in the range of values. This leads to the inequalities
\begin{eqnarray}
N_1\le &2\cdot\#\Big\{a,b\in\Z : 1\le a\le m, 1\le b\le n, &(a,m)=(b,n)=1,\label{Q_poverlap4}\\
 & &\left|\frac{a}{m}-\frac{b}{n}\right|_p\le\Delta, a/m>b/n\Big\},~\text{ and }\nonumber\\
N_1+N_3\ge &\#\Big\{a,b\in\Z : 1\le a\le m, 1\le b\le n, &(a,m)=(b,n)=1,\nonumber\\
 & &\left|\frac{a}{m}-\frac{b}{n}\right|_p\le\Delta, \frac{a}{m}>\frac{b}{n}\Big\}\label{Q_poverlap5}\\
&+~\#\Big\{a,b\in\Z : 1\le a\le m, 1\le b\le n, &(a,m)=(b,n)=1, \nonumber\\
 & &\left|\frac{a}{m}-\frac{b}{n}+1\right|_p\le\Delta, \frac{a}{m}<\frac{b}{n}\Big\}.\nonumber
\end{eqnarray}
Of course in (\ref{Q_poverlap5}) if $a/m>b/n$ then $a/m-b/n\in (0,1)$ while if $a/m<b/n$ then $a/m-b/n+1\in (0,1)$. Thus for a lower bound on $N_1+N_3$ we may use Theorem \ref{ZGaddthm} to consider all combinations $a/m-b/n$ at once, by choosing our representatives $\gamma$ from $(0,1)$ and counting how many have $|\gamma|_p\le\Delta.$ In this way we find that
\begin{align*}
\mu_p(\E_m(\psi)\cap \E_n(\psi))\ge 2\delta\varphi (d)\sum_{e|d'}c(d',e)\#\left\{a\in\N :a\le\frac{mn}{de}, \left(a,\frac{mn}{de}\right)=1, |a|_p\le\Delta\right\}.
\end{align*}
If we write $\Delta=p^{-N}$ then we have that $|a|_p\le\Delta$ if and only if $p^N|a$, and so
\begin{align*}
\mu_p(\E_m(\psi)\cap \E_n(\psi))&\ge 2\delta\varphi (d)\sum_{e|d'}c(d',e)\#\left\{a\in\N :a\le\frac{mn\Delta}{de}, \left(a,\frac{mn}{de}\right)=1\right\}.
\end{align*}
Connecting this with (\ref{Roverlap2}) proves the left hand inequality in the statement of the lemma.

For the other inequality notice that the quantities $a/m-b/n+1$ which appear in (\ref{Q_poverlap6}) all lie in the interval $(0,2)$. Thus for an upper bound we may use Theorem \ref{ZGaddthm} as before by choosing our representatives $\gamma$ from $(0,1)$ and counting how many have $|\gamma|_p\le\Delta$ or $|\gamma+1|_p\le\Delta$. Overestimating (\ref{Q_poverlap4}) in the same way leads to the bound
\begin{align}
\mu_p(\E_m(\psi)\cap \E_n(\psi))&\le 6\delta\varphi (d)\sum_{e|d'}c(d',e)\#\left\{a\in\N :a\le\frac{2mn}{de}, \left(a,\frac{mn}{de}\right)=1, |a|_p\le\Delta\right\}\nonumber\\
&\le 6\delta\varphi (d)\sum_{e|d'}c(d',e)\#\left\{a\in\N :a\le\frac{2mn\Delta}{de}, \left(a,\frac{mn}{de}\right)=1\right\},\label{Q_poverlap1}
\end{align}
and connecting this with (\ref{Roverlap1}) finishes the proof of the lemma.
\end{proof}
The deficiency of Lemma \ref{Q_pRoverlaplem} is that it is only applicable when $\psi (n)<1/4n$. However the following result will be sufficient to dispatch with the case of larger $\psi$.
\begin{lemma}\label{mediumpsilemma}
Suppose that $\psi$ takes values in the set $\{0,1,p^{-1},p^{-2},\ldots\}$. If $n\in\N$, $p\nmid n$, and $1/n\le\psi (n)< 1/12\varphi (n)$ then
\[\mu_p (\E_n(\psi))\ge \varphi (n)\psi (n).\]
Furthermore if $m\in\N$, $p\nmid m,$ and $1/m\le \psi (m)<1/12\varphi (m)$ then
\[\mu_p(\E_m(\psi)\cap\E_n (\psi))\ll \mu_p(\E_m(\psi))\cdot\mu_p(\E_n(\psi)),\]
and the implied constant is universal.
\end{lemma}
\begin{proof}
Write $\psi (n)=p^{-N}$. For the first part of the lemma we begin with the formula
\begin{align}
\mu_p (\E_n)&=\psi (n)\left(2\varphi (n)-A(n)\right),\label{mediumpsieqn1}\\
\intertext{where}
A(n)&=\sum_{\ell=1}^{p^N}\max\left\{0,-1+\sum_{\substack{a=-n\\(a,n)=1\\a\equiv\ell\mod p^N}}^n1\right\}.\nonumber
\end{align}
For an upper bound on $A(n)$ we have
\begin{align*}
A(n)&\le \sum_{1\le b\le 2n/p^N}\sum_{\substack{a=-n\\(a(a+bp^N),n)=1}}^n1\\
&=\sum_{1\le b\le 2n/p^N}\sum_{\substack{a=-n\\(a,n)=1}}^n\sum_{d|a+bp^N,n}\mu (d)\\
&=\sum_{1\le b\le 2n/p^N}\sum_{\substack{d|n\\(d,bp^N)=1}}\mu (d)\sum_{\substack{a=-n\\(a,n)=1\\a\equiv -bp^N\mod d}}^n 1\\
&=\sum_{1\le b\le 2n/p^N}\sum_{\substack{d|n\\(d,bp^N)=1}}\mu (d)\sum_{\substack{e|n\\(e,d)=1}}\mu (e)\sum_{\substack{c=-n/e\\ec\equiv -bp^N\mod d}}^{n/e}1\\
&=\sum_{1\le b\le 2n/p^N}\sum_{\substack{d|n\\(d,bp^N)=1}}\mu (d)\sum_{\substack{e|n\\(e,d)=1}}\frac{2n\mu (e)}{de}\\
&=2n\sum_{1\le b\le 2n/p^N}\sum_{\substack{d|n\\(d,bp^N)=1}}\frac{\mu (d)}{d}\sum_{\substack{e|n\\(e,d)=1}}\frac{\mu (e)}{e}\\
&=2\varphi (n)\sum_{1\le b\le 2n/p^N}\sum_{\substack{d|n\\(d,bp^N)=1}}\frac{\mu (d)}{d}\prod_{q|d}\left(1-\frac{1}{q}\right)^{-1}\\
&=2\varphi (n)\sum_{1\le b\le 2n/p^N}\prod_{\substack{q|n\\q\nmid bp^N}}\left(1-\frac{1}{q-1}\right)\\
&\le \frac{2\varphi (n)^2}{n}\sum_{1\le b\le 2n/p^N}\frac{b}{\varphi (b)}\\
&< \frac{12\varphi (n)^2}{p^N}.
\end{align*}
For the penultimate inequality we have used the fact that if $q|n$ then $q\nmid bp^N$ if and only if $q\nmid b$. The final inequality is an elementary result in number theory (see \cite[Lemma 2.5]{HarmanMNT}). Now since $\psi (n)<1/12\varphi (n)$ we have that $A(n)<\varphi (n)$ and the first part of the lemma follows from (\ref{mediumpsieqn1}).

For the proof of the second part of the lemma we begin by noticing that inequality (\ref{Q_poverlap1}) holds even when $\psi (n)\ge 1/4n$. This can be seen by starting from the inequality
\begin{align*}
\mu_p(\E_m(\psi)&\cap \E_n(\psi))\nonumber\\
&\le \delta\cdot\#\left\{a,b\in\Z : |a|\le m, |b|\le n,(a,m)=(b,n)=1, \left|\frac{a}{m}-\frac{b}{n}\right|_p\le\Delta\right\}
\end{align*}
and then proceeding in exactly the same way as in the proof of Lemma \ref{Q_pRoverlaplem}. By Theorem \ref{ZGaddthm} the quantity on the right hand side of (\ref{Q_poverlap1}) is equal to
\[3\delta\cdot\#\Big\{a,b\in\N : a\le m,~ b\le n, ~(a,m)=(b,n)=1,~ \left\|\frac{a}{m}-\frac{b}{n}\right\|\le2\Delta\Big\}.\]
The proof of the Lemma on p.196 of \cite{VaughanPollington} (also see the comments at the beginning of \cite[Section 4]{VaughanPollington}) shows that this is
\[\ll \lambda (A_m(\psi))\cdot\lambda(A_n (\psi)),\] and it follows from the first part of this lemma together with \cite[Equation (3)]{VaughanPollington} that
\[\lambda (A_m(\psi))\cdot\lambda(A_n (\psi))\ll \mu_p(\E_m(\psi))\cdot\mu_p(\E_n (\psi)).\]
\end{proof}

\section{Proofs of Theorems \ref{Q_pRtransferthm} and \ref{RQ_ptransferthm}}
We will now demonstrate the proofs of Theorems \ref{Q_pRtransferthm} and \ref{RQ_ptransferthm}. For the proof of Theorem \ref{Q_pRtransferthm} we first establish the following lemma.
\begin{lemma}\label{Q_pRtransferlem}
For any fixed prime $p$ if ($QIA_p$) holds then ($QIA_\infty ,\psi$) holds for any $\psi$ which satisfies (\ref{DSdivcond1}) and $\psi (n)=0$ for all $n\in p\N$.
\end{lemma}
\begin{proof}
The condition ($QIA_\infty, \psi$) is already known to hold in all cases when
\[\sum_{\substack{n\in\N\\\psi(n)>c/n}}\lambda (A_n)=\infty\]
for any positive constant $c$ (see \cite{VaughanPollington}). Therefore we may assume that $\psi$ satisfies $\psi (n)<1/4n$ for all $n\in\N$. Furthermore by contracting the component intervals of each of the sets $A_n(\psi/2)$ by a factor of at most $1/p$ we can arrange for $\psi$ to take values in the set $\{0,1,p^{-1},p^{-2},\ldots\}$. By the lemma of Cassels mentioned in the proof of Lemma \ref{zeroonelem} this contraction will not affect the measure of the set $W_\infty(\psi)$, and it will certainly not affect the divergence of (\ref{DSdivcond1}).

For the measures of the sets $A_n$ and $\E_n$ we have that
\[2\cdot\lambda (A_n(\psi/2))=\mu_p (\E_n(\psi))=2\varphi (n)\psi (n),\]
and by Lemma \ref{Q_pRoverlaplem} we then have that
\[\limsup_{N\rar\infty}\frac{\left(\sum_{n\le N}
    \lambda(A_n(\psi/2))\right)^2}{\sum_{m,n\le N}
    \lambda(A_m(\psi/2)\cap A_n(\psi/2))}\gg\limsup_{N\rar\infty}\frac{\left(\sum_{n\le N}
    \mu_p(\E_n(\psi))\right)^2}{\sum_{m,n\le N}
    \mu_p(\E_m(\psi)\cap\E_n(\psi))}>0.\]
Since this is true for all functions $\psi$ satisfying our hypotheses, we are finished.
\end{proof}
With this lemma as a stepping stone we give the proof of Theorem \ref{Q_pRtransferthm}.
\begin{proof}[Proof of Theorem \ref{Q_pRtransferthm}]
Suppose that $\psi$ is a function which satisfies the hypothesis of Theorem \ref{Q_pRtransferthm}, with $\psi (n)<1/8n$ for all $n\in\N$, and for which (\ref{DSdivcond1}) holds. Define a map $\tau:S\rar(\N\setminus p\N)$ by writing each integer $n\in S$ as $n=p^\ell m$ with $p\nmid m$ and then setting $\tau (n)=m$. Writing $S'=\tau (S)$ it follows that each integer in $m\in S'$ has at most $N+1$ preimages in $S$, the possibilities being the elements of the set $\{m,pm,p^2m,\ldots ,p^Nm\}$. We can thus choose a (not necessarily unique) non-negative integer $k\le N$ for which
\begin{equation}\label{DSdivcond2}
\sum_{m\in S'}\lambda (A_{p^km}(\psi))=\infty,
\end{equation}
and we may assume that $k>0$ since otherwise the conclusion of the theorem follows directly from Lemma \ref{Q_pRtransferlem}. Now consider the function $\psi':\N\rar\R$ defined by \[\psi'(n)=\begin{cases} \psi (p^kn)&\text{ if } n\in S',\\0 &\text{ if } n\notin S'.\end{cases}\]
First of all we have for $n\in S'$ that
\begin{align}\label{DSdivcond3}
\lambda (A_n(2p^k\psi'))=4\varphi (n)p^k\psi (p^kn)=\frac{2p^k}{\varphi (p^k)}\cdot\lambda (A_{p^kn}(\psi)),
\end{align}
and (\ref{DSdivcond2}) thus guarantees the divergence of the sum
\[\sum_{n\in\N}\lambda (A_n(2p^k\psi')).\]
Therefore by Lemma \ref{Q_pRtransferlem} we know that ($QIA_\infty , 2p^k\psi'$) holds. To finish our proof we will show that as $N\rar\infty$,
\begin{equation}\label{DSlimsupeqn1}
\frac{\left(\sum_{\substack{n\in S'\\n\le N}}
    \lambda(A_{p^kn}(\psi))\right)^2}{\sum_{\substack{m,n\in S'\\m,n\le N}}
    \lambda(A_{p^km}(\psi)\cap A_{p^kn}(\psi))}\gg_{p,k} \frac{\left(\sum_{\substack{n\in S'\\n\le N}}
    \lambda(A_n(2p^k\psi'))\right)^2}{\sum_{\substack{m,n\in S'\\m,n\le N}}
    \lambda(A_m(2p^k\psi')\cap A_n(2p^k\psi'))}.
\end{equation}
This implies quasi-independence on average in the limit as $N\rar\infty$ for the sequence of sets $\{A_{p^kn}(\psi):n\in S'\}$. In light of (\ref{DSdivcond2}) and our zero-one law it will follow that $\lambda(W_\infty(\psi))=1$.

To prove (\ref{DSlimsupeqn1}) we will again make use of the group ring identities from Theorem \ref{ZGaddthm}. Suppose that $m,n\in S'$ and write
\begin{align*}
m_0=p^km, ~n_0=p^kn, ~d=(m,n), ~\text{ and }~d_0=(m_0,n_0).
\end{align*}
Also let $d'$ and $d_0'$ be the divisors of $d$ and $d_0$ which satisfy the relevant hypothesis of Theorem \ref{ZGaddthm}. Observe that $d_0=p^kd$ and that $d_0'=p^kd'.$ Finally let
\begin{align*}
\delta =\delta(m,n)&=\min\left\{\psi' (m),\psi' (n)\right\}=\min\left\{\psi (m_0),\psi (n_0)\right\}\text{ and }\\
\Delta =\Delta(m,n)&=\max\left\{\psi' (m),\psi' (n)\right\}=\max\left\{\psi (m_0),\psi (n_0)\right\}.
\end{align*}
By the same type of analysis used in the proof of Lemma \ref{Q_pRoverlaplem} we have that
\begin{align*}
\lambda (A_{m_0}(\psi)\cap &A_{n_0}(\psi))
\le 4\delta\varphi (d_0)\sum_{e|d_0'}c(d_0',e)\#\left\{a\in\N :a\le\frac{2m_0n_0\Delta}{d_0e}, \left(a,\frac{m_0n_0}{d_0e}\right)=1\right\}\\
&=4\varphi (p^k)\delta\varphi (d)\sum_{e|d_0'}c(d_0',e)\#\left\{a\in\N :a\le\frac{2p^kmn\Delta}{de}, \left(a,\frac{p^kmn}{de}\right)=1\right\}.
\end{align*}
By partitioning the divisors $e$ of $d'$ according to the powers of $p$ which divide them we find that
\begin{align*}
\sum_{e|d_0'}&c(d_0',e)\#\left\{a\in\N :a\le\frac{2p^kmn\Delta}{de}, \left(a,\frac{p^kmn}{de}\right)=1\right\}\\
=&\sum_{\ell=0}^k\sum_{e|d'}c(p^kd',p^\ell e)\#\left\{a\in\N :a\le\frac{2p^{k-\ell}mn\Delta}{de}, \left(a,\frac{p^{k-\ell}mn}{de}\right)=1\right\}\\
=&\left(1-\frac{1}{p-1}\right)\sum_{e|d'}c(d',e)\#\left\{a\in\N :a\le\frac{2p^kmn\Delta}{de}, \left(a,\frac{p^kmn}{de}\right)=1\right\}\\
&+\sum_{\ell=1}^k\sum_{e|d'}c(d',e)\#\left\{a\in\N :a\le\frac{2p^{k-\ell}mn\Delta}{de}, \left(a,\frac{p^{k-\ell}mn}{de}\right)=1\right\}\\
\le&\left(1-\frac{1}{p-1}\right)\sum_{e|d'}c(d',e)\#\left\{a\in\N :a\le\frac{2mn(p^k\Delta)}{de}, \left(a,\frac{mn}{de}\right)=1\right\}\\
&+\sum_{\ell=1}^k\sum_{e|d'}c(d',e)\#\left\{a\in\N :a\le\frac{2mn(p^{k-\ell}\Delta)}{de}, \left(a,\frac{mn}{de}\right)=1\right\}.
\end{align*}
Now by appealing to equation (\ref{Roverlap1}) (but being careful of the slight difference in notation) we find that $\lambda (A_{m_0}(\psi)\cap A_{n_0}(\psi))$ is bounded above by
\begin{align}
&\frac{4\varphi (p^k)}{p^k}\left(1-\frac{1}{p-1}\right)(p^k\delta)\varphi (d)\sum_{e|d'}c(d',e)\#\left\{a\in\N :a\le\frac{2mn(p^k\Delta)}{de}, \left(a,\frac{mn}{de}\right)=1\right\}\nonumber\\
&\quad+\sum_{\ell=1}^k\frac{4\varphi (p^k)}{p^{k-\ell}}(p^{k-\ell}\delta)\varphi (d)\sum_{e|d'}c(d',e)\#\left\{a\in\N :a\le\frac{2mn(p^{k-\ell}\Delta)}{de}, \left(a,\frac{mn}{de}\right)=1\right\}\nonumber\\
&\le (1-2/p)\cdot\lambda (A_m(2p^k\psi')\cap A_n(2p^k\psi'))\label{Roverlap3}\\
&\qquad\qquad\qquad\qquad+\sum_{\ell=1}^kp^{\ell-1}(p-1)\cdot\lambda (A_m(2p^{k-\ell}\psi')\cap A_n(2p^{k-\ell}\psi')).\nonumber
\end{align}
Since
\[\lambda (A_m(2p^{k-\ell}\psi')\cap A_n(2p^{k-\ell}\psi'))\le\lambda (A_m(2p^k\psi')\cap A_n(2p^k\psi'))\]
for each $1\le\ell\le k$, equations (\ref{DSdivcond3}) and (\ref{Roverlap3}) imply (\ref{DSlimsupeqn1}) and thus finish the proof.
\end{proof}
Corollary \ref{Q_pRtransfercor} is a trivial consequence of Theorem \ref{Q_pRtransferthm}, so we move on immediately to the proof of Theorem \ref{RQ_ptransferthm}.
\begin{proof}[Proof of Theorem \ref{RQ_ptransferthm}]
Suppose that $\psi$ satisfies (\ref{Q_pDSdivcond1}). If condition (i) of Theorem \ref{dimonethm} is satisfied then we know by the proof of that theorem (below) that $\mu_p(W_p(\psi))=1$. Thus without loss of generality we may assume (by an application of the Borel-Cantelli Lemma) that $\psi(n)<1/n$ for all $n$. By replacing $\psi$ with $\psi/4$ if necessary we may further assume that $\psi (n)<1/4n$ for all $n$. The above mentioned result of Cassels guarantees that this contraction does not affect the measure of the set $W_p(\psi)$. Also by appealing to Lemma \ref{largepsilemma} we may assume without loss of generality that $\psi (n)=0$ for all $n\in p\N.$

Since we are working in $\Z_p$ it does not change anything on the $p-$adic side of things if we round down each of the values taken by the function $\psi$ so that its range is contained in the set $\{0,1,p^{-1},p^{-2},\ldots\}$. Then Lemma \ref{Q_pRoverlaplem} is immediately applicable and we find that
\[\mu_p(\E_m(\psi)\cap\E_n(\psi))\le \frac{3}{2}\cdot\lambda (A_m (2\psi)\cap A_n(2\psi)).\]
For the measures of the sets $A_n$ and $\E_n$ we have that
\[\lambda (A_n(2\psi))=2\cdot\mu_p (\E_n(\psi))=4\varphi (n)\psi (n),\]
and it follows that
\[\limsup_{N\rar\infty}\frac{\left(\sum_{n\le N}
    \mu_p(\E_n(\psi))\right)^2}{\sum_{m,n\le N}
    \mu_p(\E_m(\psi)\cap\E_n(\psi))}\gg\limsup_{N\rar\infty}\frac{\left(\sum_{n\le N}
    \lambda(A_n(\psi))\right)^2}{\sum_{m,n\le N}
    \lambda(A_m(\psi)\cap A_n(\psi))}.\]
Thus the fact that ($QIA_\infty ,\psi$) holds guarantees that ($QIA_p ,\psi$) holds and it follows from Lemma \ref{zeroonelem} that $\mu_p(W_p(\psi))=1$.
\end{proof}

\section{Proofs of Theorems \ref{dimonethm} and \ref{highdimthm}}\label{Q_presultssec}
Theorems \ref{dimonethm} and \ref{highdimthm} are analogues of results which are known to be true for the Duffin-Schaeffer Conjecture. For the most part we will appeal to proofs of the known results and then use Lemma \ref{Q_pRoverlaplem} to transfer them to $\Q_p$. However for parts (i) and (ii) of Theorem \ref{dimonethm} we will work directly with the overlap estimates obtained in Section \ref{overlapsec}.
\begin{proof}[Proof of Theorem \ref{dimonethm}]
As before we assume without loss of generality that $\psi (n)=0$ whenever $p|n.$ For the proof of part (i) we begin by defining $\psi':\N\rar\R$ by
\begin{equation*}
\psi'(n)=\begin{cases}0&\text{if }~\psi (n)<1/n,\\\min\{p\psi (n),1/12\varphi (n)\}&\text{else.}\end{cases}
\end{equation*}
Then we define $\psi'':\N\rar\{0,1,p^{-1},p^{-2},\ldots\}$ by rounding down the values taken by $\psi'$ by a factor less than $p$. From Lemma \ref{mediumpsilemma} it then follows that
\[\sum_{n\in\N}\mu_p(\E_n(\psi''))=\infty.\]
However there is one thing to be careful of here, which is that there could be integers $n\in\N$ for which $0<\psi''(n)<1/n$. Let $S\subseteq\N$ be the set of all such integers. If $n\in S$ then since $p/n\le p\psi (n)$ it follows that \[\min\{p\psi (n),1/12\varphi (n)\}=1/12\varphi (n)<p/n.\]
This means that $\varphi (n)>n/12p$ for all $n\in S$. If the sum
\begin{equation}\label{dimoneeqn1}
\sum_{n\in S}\mu_p(\E_n(\psi''))
\end{equation}
diverges then by replacing $\psi''$ by $\psi''/4$ (which does not affect the divergence of the sum or the measure of $W_p(\psi'')$) we may appeal to part (ii) below to conclude that $\mu_p(W_p(\psi))=\mu_p(W_p(\psi''))=1.$ Thus we may assume that (\ref{dimoneeqn1}) converges and then by applying the Borel-Cantelli Lemma we may arrange for $\psi''(n)$ to be greater than or equal to $1/n$ whenever it is non-zero. Now Lemma \ref{mediumpsilemma} readily applies to show that ($QIA_p,\psi''$) holds, which in turn guarantees that $\mu_p(W_p(\psi))=1$.

For the rest of this proof we assume without loss of generality that $\psi (n)<1/4n$ for all $n$ (i.e. by the same argument used at the beginning of the proof of Theorem \ref{RQ_ptransferthm}). It follows that $\mu_p(\E_n(\psi))=2\varphi (n)\psi (n)$.

For the proof of part (ii) we start from (\ref{Q_poverlap1}) to obtain
\begin{align*}
\mu_p(\E_m\cap\E_n)&\le 6\delta\varphi (d)\sum_{e|d'}c(d',e)\#\left\{a\in\N :a\le\frac{mn\Delta}{de}, \left(a,\frac{mn}{de}\right)=1\right\}\\
&\le 6\delta\varphi (d)\sum_{e|d'}\#\left\{a\in\N :a\le\frac{mn\Delta}{de}, \left(a,\frac{mn}{de}\right)=1\right\}\\
&=6\delta\varphi (d)\#\left\{a\in\N :a\le\frac{mn\Delta}{d}, \left(a,\frac{mn}{dd'}\right)=1\right\}\\
&\le\frac{6\delta\varphi (d)mn\Delta}{d}\\
&\le 6mn\psi (m)\psi (n).
\end{align*}
Then we have that
\begin{equation*}
\limsup_{N\rar\infty}\frac{\left(\sum_{n\le N}
    \mu_p(\E_n)\right)^2}{\sum_{m,n\le N}
    \mu_p(\E_m\cap \E_n)}\gg \limsup_{N\rar\infty}\left(\frac{\sum_{n\le N}
    \varphi (n)\psi (n)}{\sum_{n\le N}n\psi (n)}\right)^2>0,
\end{equation*}
which implies that $\mu_p (W_p(\psi))=1.$

For the proof of part (iii) first notice that
\begin{align}
\sum_{\substack{n\le N\\p\nmid n}}\varphi (n)&=\sum_{\substack{n\le N\\p\nmid n}}n\sum_{d|n}\frac{\mu (d)}{d}=\sum_{\substack{d\le N\\p\nmid d}}\mu (d)\sum_{\substack{e\le N/d\\p\nmid e}}e\nonumber\\
&=\sum_{\substack{d\le N\\p\nmid d}}\mu (d)\left(\frac{(p-1)N^2}{2pd^2}+O\left(\frac{N}{d}\right)\right)\nonumber\\
&=\frac{(p-1)N^2}{2p}\sum_{\substack{d=1\\p\nmid d}}^\infty\frac{\mu (d)}{d^2}+O\left(N\sum_{d\le N}\frac{|\mu (d)|}{d}\right)\nonumber\\
&=\frac{(p-1)N^2}{2p}(1-p^{-2})^{-1}\zeta (2)^{-1}+O(N\log N)\nonumber\\
&=\frac{3pN^2}{(p+1)\pi^2}+O(N\log N).\label{phieqn1}
\end{align}
Write $\N\setminus p\N=\{n_1<n_2<\ldots\}$. By partial summation we have that
\begin{align}
\sum_{k\le N}\varphi (n_k)\psi (n_k)=&\sum_{k\le N}n_k^{-\gamma}\varphi (n_k)n_k^\gamma\psi (n_k)\nonumber\\
=&\sum_{k\le N}(n_k^{-\gamma}\psi (n_k)-n_{k+1}^{-\gamma}\psi (n_{k+1}))\sum_{\ell=1}^kn_\ell^\gamma\varphi (n_\ell)\label{partsumeqn1}\\
&+n_{N+1}^{-\gamma}\psi (n_{N+1})\sum_{\ell=1}^Nn_\ell^\gamma\varphi (n_\ell).\nonumber
\end{align}
By hypothesis the sequence $\{n_k^{-\gamma}\psi(n_k)\}$ is non-increasing, and we may assume by choosing $\gamma$ larger if necessary that $\gamma\ge -1$. Then it is easy to check that (\ref{phieqn1}) implies that
\[\sum_{\ell=1}^kn_\ell^\gamma\varphi (n_\ell)\gg_{\gamma}\sum_{\ell=1}^kn_\ell^{1+\gamma}.\]
Substituting this back into (\ref{partsumeqn1}) and reversing the partial summation then gives that
\[\sum_{k\le N}\varphi (n_k)\psi (n_k)\gg \sum_{k\le N}n_k\psi (n_k),\]
so by part (ii) we have that $\mu_p(W_p (\psi))=1$.

Part (iv) is a $p-$adic analogue of the Erd\"{o}s-Vaaler Theorem (\cite{Erdos}, \cite{Vaaler}). In Section 2.4 of \cite{HarmanMNT} it is shown that if $\psi (n)\ll n^{-2}$ then ($QIA_\infty , \psi$) holds. Thus by an application of Lemma \ref{Q_pRoverlaplem} we conclude that ($QIA_p, \psi$) holds.

Part (v) is the analogue of a result recently proved in \cite{HayPolVel} (which can also be proved by using \cite[Theorem 1 (iv)]{Harman1990}). Again, the extra divergence condition here guarantees that ($QIA_\infty , \psi$) holds, which in turn guarantees that ($QIA_p, \psi$) holds.
\end{proof}
Finally we give the following proof of Theorem \ref{highdimthm}.
\begin{proof}[Proof of Theorem \ref{highdimthm}]
Pollington and Vaughan were the first to give a complete proof of the Duffin-Schaeffer Conjecture in all dimensions greater than one (\cite{VaughanPollington}). This would correspond to the choice $\ell\ge 2$ and $k=0$ in our setup, although for our presentation we require that $k>0$. We could follow Pollington and Vaughan's original line of proof but in order to maintain the consistency of ideas in this paper we choose instead to follow the proof given in Section 3.6 of \cite{HarmanMNT}. There it is shown for $\ell\ge 2$ that as $N\rar\infty$,
\begin{equation}\label{pollvaughbound1}
\sum_{m,n\le N}\lambda (A_m(\psi)\cap A_n(\psi))^\ell\ll\left(\sum_{n\le N}\lambda (A_n(\psi))^\ell\right)^2.
\end{equation}
This (in light of a known zero-one law) is enough to ensure that if the sum of measures diverges then the $\ell$-dimensional Lebesgue measure of the set
\[\{(x_1,\ldots ,x_\ell)\in(\R/\Z)^\ell:(x_1,\ldots ,x_\ell)\in (A_n (\psi))^\ell\text{ for infinitely many } n\}\]
is equal to one.

A minor modification of the proof of Theorem \ref{dimonethm}, part (i) can be used here to deal with the case where $\psi (n)\ge 1/n$. Thus for our proof we will assume without loss of generality as before that $\psi(n)<1/4n$. By first observing that
\begin{equation*}
\E_n^{\ell,p_1,\ldots , p_k}(\psi)=(A_n(\psi))^\ell\times\prod_{i=1}^k\E_n^{0,p_i}(\psi),
\end{equation*}
we find that
\begin{align*}
\mu_{\ell,p_1,\ldots , p_k}(\E_n^{\ell,p_1,\ldots , p_k}(\psi))&=\lambda(A_n(\psi))^\ell\times\prod_{i=1}^k\mu_{p_i}(\E_n^{0,p_i}(\psi))\ge\frac{\lambda(A_n(2\psi))^{\ell+k}}{2^{\ell+k}p_1\cdots p_k}\\
\intertext{and} \mu_{\ell,p_1,\ldots , p_k}(\E_m^{\ell,p_1,\ldots , p_k}(\psi)\cap\E_n^{\ell,p_1,\ldots , p_k}(\psi))&=\lambda(A_m(\psi)\cap A_n(\psi))^\ell\times\prod_{i=1}^k\mu_{p_i}(\E_m^{0,p_i}(\psi)\cap\E_n^{0,p_i}(\psi)).
\end{align*}
Now suppose that $m,n\in\N$ and that $p_i\nmid m,n$ for any $1\le i\le k$ (as in the proof of Lemma \ref{zeroonelem} there is no loss in generality in assuming this). For each $1\le i\le k$ let $M_i$ and $N_i$ be the unique integers which satisfy
\begin{equation*}
\psi (m)\in [p_i^{-M_i},p_i^{1-M_i})~\text{ and }~\psi (n)\in [p_i^{-N_i},p_i^{1-N_i}).
\end{equation*}
Then for each $1\le i\le k$ we have
\[\E_m^{0,p_i}(\psi)=\E_m^{0,p_i}(p_i^{-M_i})~\text{ and }~\E_n^{0,p_i}(\psi)=\E_n^{0,p_i}(p_i^{-N_i}),\]
so Lemma \ref{Q_pRoverlaplem} gives that
\begin{equation*}
\mu_{p_i}(\E_m^{0,p_i}(\psi)\cap\E_n^{0,p_i}(\psi))\le\frac{3}{2}\cdot\lambda (A_m(2p_i^{-M_i})\cap A_n (2p_i^{-N_i}))\le\frac{3}{2}\cdot\lambda (A_m(2\psi)\cap A_n(2\psi)).
\end{equation*}
Putting this all together and appealing to (\ref{pollvaughbound1}) gives us
\begin{align*}
\sum_{m,n\le N}\mu_{\ell,p_1,\ldots , p_k}(\E_m^{\ell,p_1,\ldots , p_k}(\psi)\cap\E_n^{\ell,p_1,\ldots , p_k}(\psi))&\ll\sum_{m,n\le N}\lambda (A_m(2\psi)\cap A_n(2\psi))^{\ell+k}\\
&\ll\left(\sum_{n\le N}\lambda (A_n(2\psi))^{\ell+k}\right)^2\\
&\ll\left(\sum_{n\le N}\mu_{\ell,p_1,\ldots , p_k}(\E_n^{\ell,p_1,\ldots , p_k}(\psi))\right)^2,
\end{align*}
as $N\rar\infty$. This is quasi-independence on average in the limit as $N\rar\infty$ and by a standard probabilistic argument (for example Lemma 2.3 of \cite{HarmanMNT}) together with our zero-one law we then have that $\mu_{\ell, p_1,\ldots,p_k}(W_{\ell, p_1,\ldots,p_k}(\psi))=1$.
\end{proof}

\section{Hausdorff measure generalizations}\label{hausdorffsection}
There are natural generalizations of the Duffin-Schaeffer Conjecture in which one considers the metric theory of approximation on $\R/\Z$ with Lebesgue measure replaced by the Hausdorff measure associated with a dimension function. We will explain the details of this below but for now we simply wish to observe that the Hausdorff measure conjectures seem at first to be refinements of the classical conjecture. However it was recently proved by V.~Beresnevich and S.~Velani that if the Lebesgue measure version of the Duffin-Schaeffer Conjecture is true then so are all of its real Hausdorff measure generalizations (\cite{BeresnevichVelani06}). In this section we will extend these results to the $p-$adic fields.

Suppose that $(X,d)$ is a metric space with the property that for every $\rho >0$ the space can be covered by a countable collection of balls with diameters less than $\rho$. If $F\subset X$ and $\rho >0$ then a collection of balls $\{B_n\}_{n\in\N}$ in $X$ with diameters $\{d_n\}_{n\in\N}$ is called a \emph{$\rho-$cover} of $F$ if $d_n\le\rho$ for all $n$ and if
\[F\subseteq\bigcup_{n\in\N}B_n.\]
A function $f:\R^+\rar\R^+$ is called a \emph{dimension function} if it is non-decreasing, continuous, and satisfies $f(r)\rar 0$ as $r\rar 0^+$. For every dimension function $f$ and for every $\rho>0$ and $F\subseteq X$ we may define
\[\HH^f_\rho(F)=\inf_{\substack{\rho\text{-covers}\\\text{of }F}}\left(\sum_{n\in\N}f(d_n)\right)\]
where the infimum is taken over all $\rho-$covers of $F$ as above. The \emph{Hausdorff $f-$measure} of $F$ is then defined to be
\[\HH^f(F)=\lim_{\rho\rar 0^+}\HH^f_\rho(F).\]
The function $\HH^f$ is a metric outer measure on $(X,d)$ and its restriction to the Borel subsets of $X$ (or more generally the Carath\'{e}odory-measurable subsets of $X$) is a measure. When $f(r)=r^m$ for some $m\ge 0$ then $\HH^f$ is usually denoted by $\HH^m$.

The following examples illustrate some points which are important to us:
\begin{itemize}
\item[(i)] For $\ell\in\N$ let $X=\R^\ell$, equipped with the Euclidean norm, and let $f(r)=r^\ell$. It is easy to see that $\HH^f$ corresponds with a constant multiple of the $\ell-$dimensional Lebesgue measure on $\R^\ell$. In the special case when $\ell=1$ we have that $\HH^f=\lambda$.
\item[(ii)] For $\ell\in\N$ let $X=\R^\ell$, equipped with the sup-norm, and let $f(r)=r^\ell$. In this case $\HH^f=\lambda^\ell$.
\item[(iii)] Let $X=\Q_p$, equipped with the $p-$adic norm, and let $f(r)=r$. It follows from the strong triangle inequality that the diameter of a ball in $\Q_p$ of radius $p^{-N}$ is equal to $p^{-N}$, and so by comparing definitions we have that $\HH^f=\mu_p$.
\item[(iv)] Let $\ell$ be a non-negative integer, $k\in\N$, and $p_1,\ldots , p_k$ primes. Also let $X=(\R/\Z)^\ell\times\Z_{p_1}\times\cdots\times\Z_{p_k}$, equipped with the norm $|~|_{\ell,p_1,\ldots ,p_k}$. We will show in Lemma \ref{measequivlem} below that $\HH^{\ell+k}$ is a constant multiple of $\mu_{\ell,p_1,\ldots , p_k}$.
\end{itemize}

The Hausdorff measure version of the Duffin-Schaeffer Conjecture presented in \cite{BeresnevichVelani06} is a special case of the following conjecture, which we formulate using the same notation as in Theorem \ref{highdimthm} above.
\begin{conjecture}\label{HpDSC}
Let $\ell$ and $k$ be non-negative integers whose sum is greater than zero, and if $k>0$ let $p_1,\ldots ,p_k$ be (not necessarily distinct) primes. Then for any dimension function $f$ with the property that $f(r)/r^{\ell+k}$ is monotonic, we have that
\[\HH^f(W_{\ell,p_1,\ldots ,p_k}(\psi))=\HH^f\left((\R/\Z)^\ell\times\Z_{p_1}\times\cdots\times\Z_{p_k}\right)\]
if and only if
\[\sum_{n\in\N}f(\psi (n))\varphi (n)^{\ell+k}=\infty.\]
\end{conjecture}
The condition that $f(r)/r^{\ell+k}$ be monotonic is justified by the fact that the space $(\R/\Z)^\ell\times\Z_{p_1}\times\cdots\times\Z_{p_k}$ has Hausdorff dimension $\ell+k$, and it is consistent with the conditions of \cite[Conjecture 2]{BeresnevichVelani06}. In fact since we are working in an $\ell+k$ dimensional space we will always assume without loss of generality that $f(r)/r^{\ell+k}$ does not tend to zero as $r\rar 0^+$. For a more detailed explanation of why we can do this see \cite[Lemma 1]{BeresnevichVelani06}.

As before one direction of Conjecture \ref{HpDSC} is relatively easy to prove. If we suppose that \[\sum_{n\in\N}f(\psi (n))\varphi (n)^{\ell+k}<\infty\]
then it is clear that we may assume that $\psi (n)\rar 0$ as $n\rar\infty$, and for any $\rho>0$ we may choose $N_0\in\N$ large enough that $\psi (n)<\rho$ for all $n\ge N_0$. Thus for any $N\ge N_0$ the component intervals of the sets \[\E_n^{\ell,p_1,\ldots , p_k}(\psi),\quad n\ge N,\] form a $\rho-$cover of $W_{\ell,p_1,\ldots , p_k}(\psi)$. This means that
\[\HH^f_\rho(W_{\ell,p_1,\ldots , p_k}(\psi))\le\sum_{n=N}^\infty f(\psi (n))(2\varphi (n))^{\ell+k}\] for all $N$. Taking the limit as $N\rar\infty$ we find that $\HH^f_\rho(W_{\ell,p_1,\ldots , p_k}(\psi))=0$ for all $\rho>0$, which in turn implies that $\HH^f(W_{\ell,p_1,\ldots , p_k}(\psi))=0.$

One of the theorems proved in \cite{BeresnevichVelani06} is that Conjecture \ref{HpDSC} is true whenever $\ell > 1$ and $k=0$. We will extend this result in the following way.
\begin{theorem}\label{Q_pMTPthm1}
Conjecture \ref{HpDSC} is true whenever $\ell+k> 1$.
\end{theorem}
Furthermore in \cite{BeresnevichVelani06} is was proved that if the Duffin-Schaeffer Conjecture is true then Conjecture \ref{HpDSC} is true when $\ell=1$ and $k=0$. To cover the remaining case when $\ell=0$ and $k=1$ we provide the following theorem.
\begin{theorem}\label{Q_pMTPthm2}
If Conjecture \ref{Q_pDSC} is true then Conjecture \ref{HpDSC} is true whenever $\ell=0$ and $k=1$. It follows that if ($QIA_\infty$) holds then Conjecture \ref{HpDSC} is true.
\end{theorem}
The proofs of Theorems \ref{Q_pMTPthm1} and \ref{Q_pMTPthm2} depend crucially on \cite[Theorem 3]{BeresnevichVelani06}, a result which we will refer to as the Mass Transference Principle. In order to present it suppose that $(X,d)$ is a metric space, that $m\in\R^+$, and that $f$ is a dimension function. Then for each ball $B=B(x,r)\subseteq X$ we define the ball $B^f\subseteq X$ by \[B^f=B\left(x,f(r)^{1/m}\right).\] The result which we need is the following.
\begin{MTP}\cite[Theorem 3]{BeresnevichVelani06}\label{MTPthm}
Let $\{B_n\}_{n\in\N}$ be a sequence of balls in $X$ with radii tending to zero as $n\rar \infty$. Let $f$ be a dimension function such that $f(x)/x^m$ is monotonic and suppose that for any ball $B\subseteq X$
\[\HH^m\left(B\cap\limsup_{n\rar\infty}B_n^f\right)=\HH^m(B).\]
Then for any ball $B\subseteq X$ we have that
\[\HH^f\left(B\cap\limsup_{n\rar\infty}B_n\right)=\HH^f(B).\]
\end{MTP}
With the aid of the Mass Transference Principle, Theorem \ref{Q_pMTPthm2} is now within easy reach.
\begin{proof}[Proof of Theorem \ref{Q_pMTPthm2}]
Let us write $p_1=p$. If $\psi (n)$ does not tend to zero as $n\rar\infty$ then by Lemma \ref{largepsilemma} we have that $W_p(\psi)=\Z_p$ and both Conjectures \ref{Q_pDSC} and \ref{HpDSC} are trivial. Therefore we assume that $\psi (n)\rar 0$ as $n\rar\infty$ and that
\[\sum_{n\in\N}f(\psi(n))\varphi (n)=\infty.\]
Also as before it will suffice to prove the theorem under the assumption that $\psi (n)=0$ whenever $p|n$.

By Lemma \ref{mediumpsilemma} for any $n\in\N$ we have that
\[\mu_p(\E_n(f(\psi (n))))\ge\min\left\{\frac{f(\psi(n))\varphi (n)}{p},\frac{1}{12}\right\}\]
and so
\[\sum_{n\in\N}\mu_p(\E_n(f(\psi(n))))=\infty.\]
Upon noting that $\HH^1=\mu_p$ Conjecture \ref{HpDSC} follows immediately from Conjecture \ref{Q_pDSC} and the Mass Transference Principle with $m=1$.

The second statement Theorem \ref{Q_pMTPthm2} then follows from Theorems \ref{RQ_ptransferthm} and \ref{Q_pMTPthm1} and \cite[Theorem 1]{BeresnevichVelani06}.
\end{proof}
The proof of Theorem \ref{Q_pMTPthm1} will use the same ideas but first we need the following lemma.
\begin{lemma}\label{measequivlem}
Suppose that $\ell$ is a non-negative integer, that $k\in\N$, and that $p_1,\ldots , p_k$ are (not necessarily distinct) primes. Let $\HH^{\ell+k}$ be $\ell+k$ dimensional Hausdorff measure on the space $(\R/\Z)^\ell\times\Z_{p_1}\times\cdots\times\Z_{p_k}$, equipped with the sup-norm $|~|_{\ell,p_1,\ldots ,p_k}$. Then $\HH^{\ell+k}$ is a constant multiple of $\mu_{\ell,p_1,\ldots , p_k}.$
\end{lemma}
\begin{proof}
Suppose that $r>0$ and that $B=B((x_1,\ldots ,x_{\ell+k}),r)$ is a ball in $(\R/\Z)^\ell\times\Z_{p_1}\times\cdots\times\Z_{p_k}$. By letting $m_1,\ldots , m_k$ be the unique non-negative integers for which
\[B=\prod_{i=1}^\ell B(x_i,r)\times\prod_{i=\ell+1}^{\ell+k}B(x_i,p_i^{-m_{i-\ell}})\]
we find that
\[\mu_{\ell,p_1,\ldots ,p_k}(B)=\frac{r^\ell}{p_1^{m_1}\cdots p_k^{m_k}}.\]
For each $i\in\{1,\ldots , k\}$ we also have
\[B(x_i,r)\subseteq B(x_i,p_i^{1-m_i})\]
which guarantees that $r\le p_i^{1-m_i}$. Thus
\[\frac{r^{\ell+k}}{p_1\cdots p_k}\le\mu_{\ell,p_1,\ldots ,p_k}(B).\]
By appealing to the definition of Hausdorff measure we have that
\begin{equation}\label{Hmucomparisoneqn}
\HH^{\ell+k}(F)\le p_1\cdots p_k\cdot\mu_{\ell,p_1,\ldots ,p_k}(F)
\end{equation}
for any measurable set $F\subseteq (\R/\Z)^\ell\times\Z_{p_1}\times\cdots\times\Z_{p_k}$,
and in particular that \[\HH^{\ell+k}((\R/\Z)^\ell\times\Z_{p_1}\times\cdots\times\Z_{p_k})\le p_1\cdots p_k<\infty.\]
Thus both $\HH^{\ell+k}$ and $\mu_{\ell,p_1,\ldots , p_k}$ are regular translation invariant measures on the compact group $(\R/\Z)^\ell\times\Z_{p_1}\times\cdots\times\Z_{p_k}$. Therefore by the well known theorem about uniqueness of Haar measures they must be constant multiples of each other.
\end{proof}

\begin{proof}[Proof of Theorem \ref{Q_pMTPthm1}]
As mentioned above the case when $\ell>1$ and $k=0$ was proved in \cite{BeresnevichVelani06}. Therefore we assume that $k>0$ and that $\ell+k>1$.

Now suppose that
\[\sum_{n\in\N}f(\psi(n))\varphi (n)^{\ell+k}=\infty.\]
For any $n\in\N$ we have that
\[\mu_{\ell,p_1,\ldots , p_k}\left(\E_n^{\ell,p_1,\ldots , p_k}(f(\psi))\right)\gg\min\left\{f(\psi(n))\varphi (n)^{\ell+k},1\right\}\]
and so
\[\sum_{n\in\N}\mu_{\ell,p_1,\ldots , p_k}\left(\E_n^{\ell,p_1,\ldots , p_k}(f(\psi))\right)=\infty.\]
Theorem \ref{highdimthm} then tells us that
\[\mu_{\ell, p_1,\ldots,p_k}(W_{\ell, p_1,\ldots,p_k}(\psi))=\mu_{\ell, p_1,\ldots,p_k}((\R/\Z)^\ell\times\Z_{p_1}\times\cdots\times\Z_{p_k})\]
and by Lemma \ref{measequivlem} this implies that
\[\HH^{\ell+k}(W_{\ell, p_1,\ldots,p_k}(\psi))=\HH^{\ell+k}((\R/\Z)^\ell\times\Z_{p_1}\times\cdots\times\Z_{p_k}).\]
Finally we finish the proof with the Mass Transference Principle, which allows us to conclude that
\[\HH^f(W_{\ell,p_1,\ldots ,p_k}(\psi))=\HH^f\left((\R/\Z)^\ell\times\Z_{p_1}\times\cdots\times\Z_{p_k}\right).\]
\end{proof}

\section{Final Remarks}\label{conclusionsec}
In conclusion, in this paper we have formulated a natural extension of the Duffin-Schaeffer Conjecture to the fields $\Q_p$ (Conjecture \ref{Q_pDSC}). We have shown that if the variance method from probability theory can be used to solve the Duffin-Schaeffer Conjecture then the corresponding $p-$adic conjectures are true for all primes $p$ (Theorems \ref{RQ_ptransferthm} and \ref{dimonethm}). On the other hand we have shown that if the variance method can be used to verify Conjecture \ref{Q_pDSC} for even one prime $p$ then almost the entire Duffin-Schaeffer Conjecture would follow (Theorem \ref{Q_pRtransferthm}). Furthermore we have generalized Conjecture \ref{Q_pDSC} to allow simultaneous approximation in any number and combination of real and $p-$adic fields and we have solved this generalized problem in all dimensions greater than one (Theorem \ref{highdimthm}). Finally we have shown that the Mass Transference Principle can be used to deduce all of the corresponding Hausdorff measure versions of the $p-$adic conjectures, provided that the original conjectures are true (Theorems \ref{Q_pMTPthm1} and \ref{Q_pMTPthm2}).

We would like to leave the reader with three interesting unsolved problems which arose in the course of this research. We list them here in what we believe to be the order of increasing difficulty:
\begin{itemize}
\item[(i)] Suppose that the Duffin-Schaeffer Conjecture is true whenever $\psi$ is supported on the set of squarefree integers. Does it then follow that the conjecture is true for all $\psi$?
\item[(ii)] Suppose that (\ref{Q_pDSdivcond1}) is satisfied and for $x\in\Z_p$ and $N\in\N$ let \[M(N,x)=\#\{n\in\N~:~n\le N, x\in\E_n(\psi)\}.\] Under what circumstances can we obtain an asymptotic formula for $M(N,x)$ which holds for almost all $x\in\Z_p$ as $N\rar\infty$?
\item[(iii)] Does there exist a function $\psi$ for which $\lambda (W_\infty(\psi))=1$ but for which the quasi-independence on average method can not be used to prove this?
\end{itemize}
It seems that an answer to the last question would bring us closer to the heart of the Duffin-Schaeffer Conjecture.

\end{document}